\documentclass[11pt]{article}
\usepackage[utf8]{inputenc}
\usepackage[english]{babel}
\usepackage{amsmath, amssymb, amsthm, bm}

\newtheorem{thm}{Theorem}[section]
\newtheorem{prop}[thm]{Proposition}

\newtheorem{corollary}[thm]{Corollary}

\theoremstyle{definition}

\title{Orbit Closures in the Witt Algebra and its Dual Space}
\author{Martin Mygind, Aarhus University}
\date{\today}

\begin{document}
\maketitle

\begin{abstract}
Working over an algebraically closed field of characteristic $p>3$, we calculate the orbit closures in the Witt algebra $W$ under the action of its automorphism group $G$. We also outline how the same techniques can be used to determine closures of orbits of all heights except $p-1$ (in which case we only obtain a conditional statement) in the dual space $W^*$ under the induced action of $G$. As a corollary we prove that the algebra of invariants $k[W^*]^G$ is trivial. 
\end{abstract}

\section{Introduction}

For classical simple Lie algebras in positive characteristic the orbits and orbit closures under the action of the adjoint group have been completely classified (see for example $\cite{C}$, which treats the equivalent problem in the theory of algebraic groups), but the same cannot be said of the simple Lie algebras of Cartan type, where little is known. Premet gave some results on the orbit structure in the Witt-Jacobson algebras $W(n)$ in his paper $\cite{P}$, and for the simplest of these - the Witt algebra $W=W(1)$ - Yao and Shu found a set of representatives for the nilpotent orbits in $\cite{YS}$. In particular, they showed that there are infinitely many of these in $W$ (or more generally in all of the $W(n)$), as opposed to the classical case where the number of nilpotent orbits is always finite. Their paper also includes partial results on the closure relations, and we extend this work by providing the complete picture of the orbit closures in the Witt algebra. In the difficult cases these are hypersurfaces, and while it is very hard to write down the polynomials defining them explicitly (as Yao and Shu correctly predicted), one can in fact say enough about them to determine the orbit closures.   

In another direction, the automorphism group of $W$ acts on the dual space $W^\star$, and representatives for the orbits were found by Feldvoss and Nakano in $\cite{FN}$. The representation-theoretic applications of this result comes from the fact that if $L$ is an arbitrary restricted Lie algebra over an algebraically closed field of positive characteristic, then the irreducible representations of $L$ are related to the elements of $L^*$ by assigning to each module its so-called $p$-character in the dual space. So one can work with the reduced enveloping algebras $U_{L,\chi}$, $\chi\in L^*$, when trying to determine the irreducible modules, and these finite dimensional algebras are far easier to approach than the universal enveloping algebra (see $\cite{SF}$ for more on all this). In this context Feldvoss and Nakano's work on orbit representatives becomes a result on isomorphism classes of the $U_{W,\chi}$. Here we will calculate the closures in $W^\star$ of orbits of all heights except $p-1$, skipping lightly over some details, since the techniques we use are similar to those we developed for $W$. So the reader who understood the first part should be able to fill in the missing arguments in the second. For the exceptional case of height $p-1$ we obtain a conditional statement, and as an application of our results we prove that the algebra of invariants $k[W^*]^G$ is trivial, i.e., equal to $k$. 

Let $k$ be an algebraically closed field of characteristic $p>3$, and let $A(n)=k[ X_1,\ldots,X_n]/(X_1^p,\ldots,X_n^p)$ be the truncated polynomial ring in $n$ variables over $k$. The $n$'th Witt-Jacobson algebra $W(n)$ is defined as the Lie algebra of derivations of $A(n)$. These algebras constitute one of the four classes of simple Lie algebras of Cartan type - the others being the special algebras, the Hamiltonian algebras and the contact algebras - which have properties very different from those of the classical Lie algebras (see $\cite{SF}$). We will now restrict our attention to $W(1)=W$, which is also known just as the Witt algebra, but many of the following introductory remarks apply in generalized versions for arbitrary $W(n)$. 

We write $A(1)=k[X]/(X^p)$ and let $x$ denote the class of $X$ in $A(1)$. $W$ is an $A(1)$-module in a natural way, and we have a canonical basis $\{e_i\mid -1\leq i\leq p-2\}$ for $W$, where $e_i=x^{i+1}\partial$ and $\partial$ is the usual differential operator satisfying $\partial(x^i)=ix^{i-1}$. The following useful formula for $D\in W$ is easy to verify:
\begin{equation}\label{diff}
D=D(x)\partial
\end{equation}
The Lie bracket in $W$ is given by
\begin{equation*}
[e_i,e_j]=(j-i)e_{i+j}
\end{equation*}  
with $e_{i+j}=0$ if $i+j\notin\{-1,\ldots,p-2\}$ by definition. Thus we have a gradation
\begin{equation*}
W=\bigoplus_{i=-1}^{p-2}W_{i}
\end{equation*}
where $W_{i}=ke_i$. We set $W_{\geq i}=\bigoplus_{j=i}^{p-2}W_{j}$. Note also that $W$ is a \textit{restricted} Lie algebra, the $p$-map being given by raising to the $p$'th power (using the associative product in $W\subseteq\mathrm{End}(A(1))$): $w^{[p]}=w^p$ for all $w\in W$.

Let $G$ denote the automorphism group of $W$. It is well known that $G$ is a connected algebraic group of dimension $p-1$, and that we have an isomorphism $\mathrm{Aut}(A(1))\overset{\sim}{\longrightarrow}G$, $\varphi\mapsto \sigma_{\varphi}$, defined by $\sigma_{\varphi}(D)=\varphi\circ D\circ\varphi^{-1}$ for all $D\in W$. Furthermore, there is a semidirect product decomposition $G=T\ltimes U$ with $T$ isomorphic to the multiplicative group and $U$ unipotent. Elements $\sigma_\varphi\in T$ correspond to $\varphi\in\mathrm{Aut}(A(1))$ satisfying $\varphi(x)=tx$ for some $t\in k^{\star}$. We will identify $t$ and $\varphi$ in this case, and an easy calculation shows that the action of $T$ on $W$ is given by $t\cdotp e_i=t^ie_i$. Similarly, elements $\sigma_\varphi\in U$ correspond to $\varphi$ with $\varphi(x)=x+b_2x^2+\dots +b_{p-1}x^{p-1}$ for some $b_2,\ldots,b_{p-1}\in k$. Note that an element of $\mathrm{Aut}(A(1))$ is uniquely determined by its value on $x$.

\section{Orbits in the Witt algebra}

We will say that a nonzero $w$ in $W$ has degree $i$ if $w\in W_{\geq i}\backslash W_{\geq i+1}$ (this of course determines the degree uniquely). Using the semidirect decomposition described above, one shows that $G$ preserves degree, and thus it makes sense to speak of the degree of an orbit. The starting point for our calculations is the following theorem, which gives a complete set of representatives for the nonzero orbits in $W$ under the action of $G$:  
\begin{thm}\label{baner}
A set of representatives for the orbits of degree i is:
\begin{enumerate}
\item $\{ e_{-1}+ae_{p-2}\mid a\in k \}$ if $i=-1$
\item $\{ae_0\mid a\in k^*\}$ if $i=0$
\item $\{e_i+ae_{2i}\mid a\in k\}$ if $1\leq i<\frac{p-1}{2}$
\item $\{e_i\}$ if $\frac{p-1}{2}\leq i \leq p-2$
\end{enumerate}
So if $i<\frac{p-1}{2}$ the orbits of degree i are parametrized by k, while there is only one orbit if $i\geq\frac{p-1}{2}$. The dimensions of the orbits are:
\begin{enumerate}
\item[1'.] $\mathrm{dim}(G(e_{-1}+ae_{p-2}))=p-1$
\item[2'.] $\mathrm{dim}(G(ae_0))=p-2$
\item[3'.] $\mathrm{dim}(G(e_i+ae_{2i}))=p-i-2$ if $1\leq i<\frac{p-1}{2}$
\item[4'.] $\mathrm{dim}(Ge_i)=p-i-1$ if $\frac{p-1}{2}\leq i \leq p-2$ 
\end{enumerate}
\end{thm}
Cases 3, 3', 4 and 4' were taken care of in $\cite{YS}$ as these, along with $Ge_{-1}$ and $0$, account for the nilpotent orbits. The proofs of 1 and 2, as well as the corresponding dimension statements, are very similar, but we include them here anyway for the sake of completeness. First, however, some general considerations (for more details, see $\cite{YS})$ that will be used throughout the paper: Let $\sigma_\varphi\in U$, with $\varphi(x)=x+b_2x^2+\dots +b_{p-1}x^{p-1}$ and $\varphi^{-1}(x)=x+c_2x^2+\dots +c_{p-1}x^{p-1}$. It is convenient to set $b_1=c_1=1$ and $b_p=c_p=0$. Using the definition of the isomorphism $\varphi\mapsto \sigma_{\varphi}$ we get:
\begin{equation}\label{vform}
\sigma_{\varphi}(e_i)(x)=\varphi(x)^{i+1}(1+2c_2\varphi(x)+\dots +(p-1)c_{p-1}\varphi(x)^{p-2})
\end{equation}    
Write $\sigma_{\varphi}(e_i)=\sum_{j=i}^{p-2}a_je_j$, then formulas (\ref{diff}) and (\ref{vform}) tell us that $a_i=1$ and
\begin{equation}\label{abc}
a_j=(i+1)b_{j-i+1}+(j-i+1)c_{j-i+1}+g_j(b_2,\ldots,b_{j-i},c_2,\ldots,c_{j-i})
\end{equation}
for certain polynomials $g_j$ when $i<j\leq p-2$. We get another useful formula from looking at the coefficient of $x^j$ on the left hand side of the equation $\varphi(\varphi^{-1}(x))=x$:
\begin{equation}\label{bogc}
c_j=-b_j+h_j(b_2,\ldots,b_{j-1},c_2,\ldots,c_{j-1})
\end{equation}  
Here the $h_j$ are certain polynomials and $2\leq j\leq p-1$. By induction we see that $c_j$ can be expressed as a polynomial in $b_2,\ldots,b_j$, and inserting into ($\ref{abc}$) yields
\begin{equation}\label{aogb}
a_j=(2i-j)b_{j-i+1}+g_j'(b_2,\ldots,b_{j-i})
\end{equation}
Now we are ready for the
\begin{proof}[Proof of Theorem \ref{baner}]
We start with case 2: Note first that $t\cdotp e_0=e_0$ for all $t\in T$, and this easily implies that $ae_0$ and $be_0$ are in the same orbit if and only if $a=b$. Now let $w=w_0e_0+\dots+w_{p-2}e_{p-2}$ with $w_0\neq 0$. With notation as above, formula ($\ref{aogb}$) shows that we can choose $b_2,\ldots,b_{p-1}\in k$ recursively such that $\sigma_\varphi(e_0)=e_0+\frac{w_1}{w_0}e_1+\dots+\frac{w_{p-2}}{w_0}e_{p-2}$. But then $\sigma_\varphi(w_0e_0)=w$ and we are done. This argument actually shows that 
\begin{equation}\label{hej}
G(ae_0)=\{ae_0+a_1e_1+\dots+a_{p-2}e_{p-2}\mid a_1,\ldots,a_{p-2}\in k\}\cong \mathbb{A}^{p-2}
\end{equation} 
from which 2' follows. As for case 1, look at $y=y_{-1}e_{-1}+\dots+y_{p-2}e_{p-2}$ with $y_{-1}\neq 0$. Again we can find $b_2,\ldots,b_{p-1}$ such that $\sigma_\varphi(e_{-1})=e_{-1}+y_0e_0+y_{-1}y_1e_1+
\dots+y_{-1}^{p-3}y_{p-3}e_{p-3}+g_{p-2}'(b_2,\ldots,b_{p-1})e_{p-2}$. Choosing $a=y_{-1}^{p-2}y_{p-2}-g_{p-2}'(b_2,\ldots,b_{p-1})$ we get
\begin{equation*}
\sigma_\varphi(e_{-1}+ae_{p-2})=e_{-1}+y_0e_0+y_{-1}y_1e_1+
\dots+y_{-1}^{p-2}y_{p-2}e_{p-2}
\end{equation*}  
since $\sigma_\varphi(e_{p-2})=e_{p-2}$. Now we apply the element $y_{-1}^{-1}\in T$ and get $y_{-1}^{-1}\cdotp\sigma_\varphi(e_{-1}+ae_{p-2})=y$. Thus every element of degree $-1$ is in the orbit of $e_{-1}+ae_{p-2}$ for some $a$. It remains only to show that $e_{-1}+ae_{p-2}$ and $e_{-1}+be_{p-2}$ are in the same orbit if and only if $a=b$, and here one can use exactly the same method as in the proof of Proposition 3.4 in $\cite{YS}$ (it is also a consequence of the proof of case 1 in Theorem $\ref{mopho}$, starting on page 10). Similarly, one can mimic the proof of Theorem 4.1 in said paper to show that $G(e_{-1}+ae_{p-2})$ has trivial stabilizer in $G$, which implies
\begin{equation*}
\mathrm{dim}(G(e_{-1}+ae_{p-2}))=\mathrm{dim}(G)=p-1
\end{equation*}      
\end{proof}

The next theorem, one of the main theorems of the paper, provides the full picture on orbit closures in the Witt algebra:
\begin{thm}\label{mopho} 
Let $a\in k$ and define $B=\{b\in k\mid b^{p-1}=-a\}$. We have:
\begin{enumerate}
\item $\overline{G(e_{-1}+ae_{p-2})}= \begin{cases} G(e_{-1}+ae_{p-2}) \ \cup \ (\bigcup_{b\in B} G(be_0)) & \text{if $a\neq 0$} \\
Ge_{-1}\cup W_{\geq 1} & \text{if $a=0$}
\end{cases}$
\item $\overline{G(ae_0)}=G(ae_0)$
\item $\overline{G(e_i+ae_{2i})}=G(e_i+ae_{2i})\cup W_{\geq i+2}$ if $1\leq i<\frac{p-1}{2}$ 
\item $\overline{Ge_i}=W_{\geq i}$ if $\frac{p-1}{2}\leq i \leq p-2$  
\end{enumerate}
\end{thm}
Cases 2 and 4 are more or less trivial: In the former case, equation ($\ref{hej}$) shows that $G(ae_0)$ is a Zariski-closed subset of $W$ for all $a\in k$, and in the latter case we must have $Ge_i=W_{\geq i}\backslash W_{\geq i+1}$, which easily implies $\overline{Ge_i}=W_{\geq i}$. For orbits of type 1 or 3 we need to work considerably harder: Assume first that $1\leq i<\frac{p-1}{2}$. The case $i=1$ turns out to be degenerate, so we will save that for later and also assume $i\neq 1$. The following proposition provides further information about the action of $U$ on $e_i+ae_{2i}$. Note that we will sometimes represent an element $w\in W$ by its coordinates $(w_{-1},\dots,w_{p-2})$ with respect to the basis $\{ e_i\}$.
\begin{prop}\label{c}
\begin{equation*}
U(e_i+ae_{2i})=\left\{\begin{pmatrix}
0\\ \vdots\\0\\1\\a_{i+1}\\ \vdots \\a_{2i-1}\\f(a_{i+1},\ldots,a_{2i-1})+a\\a_{2i+1}\\ \vdots\\a_{p-2}
\end{pmatrix} \right\}
\end{equation*}
Here $a_{i+1},\ldots,\hat{a}_{2i},\ldots,a_{p-2}\in k$, and $f\in k[X_{i+1},\dots,X_{2i-1}]$ is a polynomial with the following properties:
\begin{enumerate}
\item If $X_{i+1}^{\alpha_{i+1}}\dots X_{2i-1}^{\alpha_{2i-1}}$ is a monomial appearing in $f$ with nonzero coefficient, then 
\begin{equation*}
\sum_{j=i+1}^{2i-1}(j-i)\alpha_j=i
\end{equation*}
\item The (usual) degree of $f$ is at most $i$, and the component in $f$ of degree i is $cX_{i+1}^i$ for some $c\in k$.
\end{enumerate}
\end{prop}
Note that the proposition \textit{does not} say that $c\neq0$. This actually turns out to be the case - and of crucial importance - but we are not able to prove it yet.
\begin{proof}[Proof of Proposition $\ref{c}$]
We use the notation from the discussion preceding the proof of Theorem $\ref{baner}$.
Because of the bijection $\varphi\mapsto \sigma_\varphi$ we can consider $U$ as a variety isomorphic to $\mathbb{A}^{p-2}$ with coordinate functions $
b_2,\dots b_{p-1}$. Grade the polynomial ring $k[U]$ by setting $\mathrm{deg}(b_s)=s-1$ for $2\leq s\leq p-1$. Looking closer at the equation $\varphi(\varphi^{-1}(x))=x$ we can refine formula ($\ref{bogc}$) a bit to get $c_2=-b_2$ and
\begin{equation*}
c_n=-b_n+\sum_{s=2}^{n-1}c_s\sum_{n_1+\ldots +n_s=n}b_{n_1}\ldots b_{n_s}
\end{equation*}
for $n>2$. From this it follows easily by induction that $c_2,\dots,c_{p-1}\in k[U]$, and that $c_s$ is homogeneous of degree $s-1$. Equation ($\ref{vform}$) shows that $a_j$, $i<j\leq p-2$, can be written as a sum of terms $b_{s_1}\ldots b_{s_{i+1}}c_lb_{t_1}\ldots b_{t_{l-1}}$ where $s_1+\dots+s_{i+1}+t_1+\dots+t_{l-1}=j+1$. Since such a term has degree 
\begin{equation*}
\sum_{n=1}^{i+1}(s_n-1)+(l-1)+\sum_{m=1}^{l-1}(t_m-1)=j-i
\end{equation*}
we see that $a_j\in k[U]$ is homogeneous of degree $j-i$. Furthermore, changing the indices in ($\ref{aogb}$) gives us $a_{s+i-1}=(i-s+1)b_s+h(b_2,\ldots,b_{s-1})$ for $2\leq s\leq i$ and some polynomial $h$, from which it follows by induction that
\begin{equation}\label{juno}
k[a_{i+1},\dots,a_{s+i-1}]=k[b_2,\dots,b_s]
\end{equation}
Now let $\sigma_\varphi(e_i+ae_{2i})=(0,\ldots,0,1,a_{i+1},\ldots,a_{p-2})$ (this is a slight abuse of notation, but notice that the new $a_j$ agree with the old when $i<j<2i$). Equation ($\ref{aogb}$) shows that we can recursively choose $b_2,\ldots,b_i$ to make $a_{i+1},\ldots,a_{2i-1}$ attain any set of values in $k$. But $a_{2i}=g_{2i}'(b_2,\ldots,b_i)+a$, and expressing $b_2,\ldots,b_i$ as polynomials in $a_{i+1},\ldots,a_{2i-1}$ - which is possible because of ($\ref{juno}$) - we get $a_{2i}=f(a_{i+1},\ldots,a_{2i-1})+a$ for some polynomial $f\in k[X_{i+1},\dots,X_{2i-1}]$. The coordinates $a_{2i+1},\ldots,a_{p-2}$ can again be assigned arbitrary values by choosing $b_{i+1},\ldots,b_{p-1}$ accordingly. As for properties 1 and 2, we have shown that $f(a_{i+1},\dots,a_{2i-1})$ is homogeneous of degree $2i-i=i$ when considered as an element of $k[U]$. But since each of the $a_j$ is homogeneous of degree $j-i$ we get 1, from which 2 follows directly.  
\end{proof}

Write $f=f_0+\dots+f_i$ where $f_j$ is the component in $f$ of degree $j$. In particular we have $f_i=cX_{i+1}^i$. Now we are ready to describe the $G$-orbit of $e_i+ae_{2i}$:
\begin{prop}\label{end} 
\begin{equation}\label{grim}
G(e_i+ae_{2i})=\left\{ \begin{pmatrix}
0\\ \vdots\\ 0\\ b_i\\ \vdots \\b_{2i-1}\\ \\ \sum_{j=0}^i\frac{f_j(b_{i+1},\dots,b_{2i-1})}{b_i^{j-1}}+ab_i^2\\ \\b_{2i+1}\\ \vdots\\b_{p-2}
\end{pmatrix} \right\}
\end{equation}
Here $b_i,\ldots,\hat{b}_{2i},\ldots,b_{p-2}\in k$ with $b_i\neq0$.
\end{prop}
\begin{proof}
We have $G(e_i+ae_{2i})=T(U(e_i+ae_{2i}))$, so let us look at the action of $t\in T$ on $w\in U(e_i+ae_{2i})$:
\begin{equation*}
t\cdotp w=\begin{pmatrix}
0\\ \vdots\\ 0\\ t^i\\t^{i+1}a_{i+1}\\ \vdots \\t^{2i}(f(a_{i+1},\dots,a_{2i-1})+a)\\ \vdots\\t^{p-2}a_{p-2}
\end{pmatrix}
\end{equation*}
for some $a_{i+1},\ldots,\hat{a}_{2i},\ldots,a_{p-2}\in k$. It follows from property 1 in Proposition $\ref{c}$ that
\begin{equation}\label{n}
t^if(a_{i+1},\dots,a_{2i-1})=f(ta_{i+1},t^2a_{i+2},\dots,t^{i-1}a_{2i-1})
\end{equation} 
Writing $b_i=t^i,b_{i+1}=t^{i+1}a_{i+1},\ldots,b_{p-2}=t^{p-2}a_{p-2}$ (skipping $b_{2i}$) and using ($\ref{n}$), the $2i$'th coordinate in $t\cdotp w$ becomes
\begin{align*}
&t^{2i}(f(a_{i+1},\dots,a_{2i-1})+a)=\\&b_if(\frac{b_{i+1}}{b_i},\dots,\frac{b_{2i-1}}{b_i})+ab_i^2=\sum_{j=0}^i\frac{f_j(b_{i+1},\dots,b_{2i-1})}{b_i^{j-1}}+ab_i^2
\end{align*} 
So we get a point like on the right hand side of ($\ref{grim}$). By simply reversing the process we see that every point of this kind is in $G(e_i+ae_{2i})$.   
\end{proof}
We will now ignore the first $i+1$ coordinates (since they are all zero anyway) and consider $G(e_i+ae_{2i})$ as a subset of $\mathbb{A}^{p-i-1}$. A general remark: for any polynomial $f$ in $n$ variables, we write $V(f)$ for the zero set of $f$ in $\mathbb{A}^n$. With this notation we have:
\begin{corollary}\label{kor}
\begin{align*}
G(e_i&+ae_{2i})=\\& \{w\in V(X_{2i}X_i^{i-1}-cX_{i+1}^i-\sum_{j=0}^{i-1}X_i^{i-j}f_j-aX_i^{i+1})\mid w_i\neq 0\}
\end{align*}
\begin{proof}
This follows directly from Proposition $\ref{end}$.
\end{proof}
\end{corollary}
We are almost ready to determine the orbit closure, but we need the following simple \textbf{algebraic geometric fact}: If $f\in k[X_1,\ldots,X_n]$ is a polynomial satisfying $X_j\nmid f$ for some $j\in\{1,\ldots,n\}$ and $A=\{x\in V(f)\mid x_j\neq 0\}$, then $\overline{A}=V(f)$. To prove this it is enough to show that $A$ intersects every component of $V(f)$, so assume there is a component $V(f')$ of $V(f)$ (with $f'$ an irreducible polynomial dividing $f$) such that $A\cap V(f')=\emptyset$. Then $x_j=0$ for every $x\in V(f')$, and we must have $V(f')\subseteq V(X_j)$. But this means $(X_j)\subseteq (f')$ and so $f'$ divides $X_j$, which can only happen if $f'=\alpha X_j$ for some $\alpha \neq 0$, a contradiction. Another result that will be used several times throughout the paper should also be mentioned (even though it is probably familiar to the reader): The closure of an orbit is a union of the orbit itself and certain other orbits of strictly lower dimension (\cite{H}, Proposition 8.3).     
\begin{prop}\label{fini}
\begin{equation*}
\overline{G(e_i+ae_{2i})}=V(X_{2i}X_i^{i-1}-cX_{i+1}^i-\sum_{j=0}^{i-1}X_i^{i-j}f_j-aX_i^{i+1})
\end{equation*}
and $c\neq 0$.
\end{prop}  
\begin{proof}
Let $g$ denote the polynomial on the right hand side, and write $g=X_i^s g'$ with $X_i\nmid g'$. It follows easily from Corollary $\ref{kor}$ and the aforementioned algebraic geometric fact that $\overline{G(e_i+ae_{2i})}=V(g')$. We prove the proposition using descending induction in $i$: Assume first that $i=\frac{p-1}{2}-1$. To reach a contradiction we also assume $c=0$. Then $g'$ cannot contain a term of the form $\beta X_{i+1}^m$ (for some $\beta\neq 0,m>0$) because of property 1 in Proposition $\ref{c}$, so we must have a point $w\in V(g')$ with $w_i=0$, $w_{i+1}\neq 0$. It follows that $\overline{G(e_{\frac{p-1}{2}})}\subsetneq \overline{G(e_i+ae_{2i})}$, but this is impossible since the dimension of the two varieties is the same (Theorem $\ref{baner}$). Therefore $c\neq 0$ and $\overline{G(e_i+ae_{2i})}=V(g')=V(g)$.

Assume now that the claim is true for $i+1$, and that $c=0$ when we calculate $G(e_i+ae_{2i})$. Exactly as before we can find a $w\in V(g')$ satisfying $w_i=0$, $w_{i+1}\neq 0$. This means that there exists $a'\in k$ such that $\overline{G(e_{i+1}+a'e_{2(i+1)})}\subseteq \overline{G(e_i+ae_{2i})}$, and by induction we know      \begin{align*}
&\overline{G(e_{i+1}+a'e_{2(i+1)})}=\\& V(X_{2(i+1)}X_{i+1}^i-c'X_{i+2}^{i+1}-\sum_{j=0}^iX_{i+1}^{i+1-j}f_j'-a'X_{i+1}^{i+2})
\end{align*}
where the $f_j'$ are certain homogeneous polynomials in $X_{i+2},\dots,X_{2i+1}$ and $c'\neq 0$. Let $h$ denote the polynomial on the right hand side. Notice that $h$ is irreducible: It is a first degree polynomial in $X_{2(i+1)}$, and the coefficients have no common factors, thanks to $c'$ being different from zero. If we set $g''=g'(0,X_{i+1},\ldots,X_{2i})$, then $\overline{G(e_{i+1}+a'e_{2(i+1)})}\subseteq \overline{G(e_i+ae_{2i})}$ means that $V(h)\subseteq V(g'')$, so $h$ divides $g''$. By looking at the degree of the two polynomials in the variable $X_{i+1}$ this is seen to be impossible, and so $c\neq 0$ and we are done.   
\end{proof}
\begin{proof}[Proof of case 3 in Theorem \ref{mopho}]
Combining Corollary $\ref{kor}$ and Proposition $\ref{fini}$ we get 
\begin{equation*}
\overline{G(e_i+ae_{2i})}=G(e_i+ae_{2i})\cup \{w\in V(g)\mid w_i=0\}
\end{equation*}
Since $g(0,X_{i+1},\ldots,X_{2i})=-cX_{i+1}^i$ we see that $\{w\in V(g)\mid w_i=0\}=W_{\geq i+2}$, and so 
\begin{equation*}
\overline{G(e_i+ae_{2i})}=G(e_i+ae_{2i})\cup W_{\geq i+2}
\end{equation*}
This is true also for the degenerate case $i=1$: Here we get $f=0$ in the setup of Proposition $\ref{c}$, and inserting this in Proposition $\ref{end}$ gives us 
\begin{equation*}
G(e_1+ae_2)=\{w\in V(X_2-aX_1^2)\mid w_1\neq 0\}
\end{equation*}
Using the remark preceding Proposition $\ref{fini}$ we conclude that $\overline{G(e_1+ae_2)}=G(e_1+ae_2)\cup W_{\geq 3}$.
\end{proof}

The last case in Theorem $\ref{mopho}$ could be proved using a similar method, but there is an easier and more elegant way: 

\begin{proof}[Proof of case 1 in Theorem $\ref{mopho}$]
The statement for $a=0$ was proved in $\cite{YS}$, so assume $a\neq 0$. For an arbitrary $\lambda \in k$ we define:
\begin{equation*}
W(\lambda)=\{w\in W\mid w^{[p]}=\lambda w\}\backslash\{0\}
\end{equation*}
The idea of the proof is to describe the set $W(-a)$ in two different ways. First we claim that
\begin{equation}\label{see}
W(-a)=G(e_{-1}+ae_{p-2}) \ \cup \ (\bigcup_{b\in B} G(be_0))
\end{equation}
To show this, note first that the sets $W(\lambda)$ are $G$-invariant: For $w\in W(\lambda)$ and $\sigma\in G$ we have:
\begin{equation*}
\sigma(w)^{[p]}=\sigma(w^{[p]})=\sigma(\lambda w)=\lambda\sigma(w)
\end{equation*}
Here the first equality comes from the fact that $W$ is centerless (it is simple), so any automorphism respects the $p$-mapping. Set $D=e_{-1}+ae_{p-2}$. By an easy induction in $i$ we get
\begin{equation*}
D^i(x)=a\frac{(p-1)!}{(p-i)!}x^{p-i}
\end{equation*}
for $2\leq i\leq p-1$. In particular $D^{p-1}(x)=a(p-1)!x=-ax$ by Wilson's Theorem, so $D^p(x)=D(-ax)=-aD(x)$. Since elements in $W$ are uniquely determined by their value on $x$, we get $D^p=-aD$, which implies $GD\subseteq W(-a)$. Now let $b\in k$ with $b^{p-1}=-a$. One easily checks that $e_0^{[p]}=e_0$, and from this it follows that $(be_0)^{[p]}=b^pe_0^{[p]}=-a(be_0)$. So $G(be_0)\subseteq W(-a)$ and we have shown:
\begin{equation}\label{no}
W(-a)\supseteq G(e_{-1}+ae_{p-2}) \ \cup \ (\bigcup_{b\in B} G(be_0))
\end{equation}
Since $W(0)=(Ge_{-1}\cup W_{\geq 1})\backslash\{0\}$ ($\cite{YS}$, Lemma 3.1) we see that every nonzero element of $W$ is contained in some $W(\lambda)$. But $W(\lambda)\cap W(\mu)=\emptyset$ whenever $\lambda\neq \mu$, and this implies equality in ($\ref{no}$).

Now consider $w\in W$ as an endomorphism of $A(1)$, and write $\mathrm{char}(w)$ for its characteristic polynomium. Corollary 1 in $\cite{P}$ gives us
\begin{equation*}
\mathrm{char}(w)=X^p+\varphi(w_{-1},\dots,w_{p-2})X=X(X^{p-1}+\varphi(w_{-1},\dots,w_{p-2}))
\end{equation*}
where $\varphi$ is a nonzero homogeneous polynomial of degree $p-1$. Our second claim is that
\begin{equation}\label{good}
W(-a)=V(\varphi-a)
\end{equation}  
Assume first that $w\in W(-a)$. Then $w$ is semisimple considered as an endomorphism of $A(1)$ ($\cite{SF}$, Proposition 3.3). In particular $w$ has a nonzero eigenvalue $\mu\in k^*$, and using the definition of $W(-a)$ one checks that $\mu^{p-1}=-a$. Inserting $\mu$ into the characteristic polynomial of $w$ gives us $w\in V(\varphi-a)$. Next, assume this is true. Then $w$ is not nilpotent, and therefore it is contained in $W(\lambda)$ for some $\lambda\neq 0$. But then any nonzero eigenvalue $\mu$ of $w$ satisfies both $\mu^{p-1}=\lambda$ and $\mu^{p-1}=-a$. So $\lambda=-a$ and we have shown ($\ref{good}$).

We are now ready to complete the proof: Since the $W(\lambda)$ are closed, it follows from ($\ref{see}$) that $\overline{G(e_{-1}+ae_{p-2})}\subseteq W(-a)$. Assume for some $b\in k$ with $b^{p-1}=-a$ that $G(be_0)$ is not in the closure of $G(e_{-1}+ae_{p-2})$. Then we see from ($\ref{see}$) that $G(be_0)$ must be an irreducible component of $W(-a)$, of dimension $p-2$ (Theorem $\ref{baner}$). But the description in ($\ref{good}$) shows that every component of $W(-a)$ has dimension $p-1$ (see $\cite{H}$, Theorem 3.3). This contradiction implies $\overline{G(e_{-1}+ae_{p-2})}= W(-a)$, and we are done.      
\end{proof}

\section{Orbits in the dual space $W^*$}
Let $W^*$ denote the dual space of the Witt algebra $W$, i.e., the vector space of all $k$-linear maps $W\rightarrow k$, and let $\{e_{-1}',\ldots,e_{p-2}'\}$ denote the basis of $W^*$ dual to the basis $\{e_{-1},\ldots,e_{p-2}\}$ of $W$. So $e_i'(e_j)=\delta_{ij}$ for $i,j\in \{-1,\ldots,p-2\}$. The elements of $W^*$ are called characters because of the representation theoretic applications mentioned in the introduction. The height $r(\chi)$ of a character $\chi\in W^*$ is defined as
\begin{equation*}
r(\chi)=\begin{cases}
\mathrm{min}\{i\in \{-1,\ldots,p-2\}\mid \chi_{\vert_ {W_{\geq i}}}=0\} & \text{if $\chi(e_{p-2})=0$} \\
p-1 & \text{if $\chi(e_{p-2})\neq0$}
\end{cases}
\end{equation*}
The notion of height has become standard in the literature on the Witt algebra and its representations, so we will use it in what follows, even though it would have been more in keeping with what we did for $W$ to use the height minus one. The automorphism group $G=\mathrm{Aut}(W)$ acts on $W^*$, the action being given by 
\begin{equation*}
g\cdotp\chi(w)=\chi(g^{-1}(w))
\end{equation*}
for $g\in G,\chi\in W^*,w\in W$. It is easily seen that $G$ preserves height, so we have a well-defined notion of height for the orbits as well. Now recall the notation introduced right after Theorem $\ref{baner}$: For $\sigma_\varphi\in U$ we write $\varphi(x)=x+b_2x^2+\dots +b_{p-1}x^{p-1}$ and $\varphi^{-1}(x)=x+c_2x^2+\dots +c_{p-1}x^{p-1}$. Using the definition of the action we get $\sigma_\varphi^{-1}(e_i')=\sum_{j=-1}^i a_je_j'$ where $a_i=1$ and $a_j$ ($-1\leq j\leq i-1$) is the coefficient of $x^{i+1}$ in $\sigma_{\varphi}(e_j)(x)$. More precisely we have (see equation ($\ref{vform}$))
\begin{equation*}
a_j=(j+1)b_{i-j+1}+(i-j+1)c_{i-j+1}+p_j(b_2,\ldots,b_{i-j},c_2,\ldots,c_{i-j})
\end{equation*}
where $-1\leq j\leq i-1$ and the $p_j$ are certain polynomials. Applying formula ($\ref{bogc}$) to express the $c_k$ in terms of the $b_k$, we get
\begin{equation}\label{ny}
a_j=(2j-i)b_{i-j+1}+p_j'(b_2,\ldots,b_{i-j})
\end{equation}
It should also be noted that the action of the torus is now given by $t\cdotp e_i'=t^{-i}e_i'$ for $t\in T$.

A complete set of representatives for the nonzero orbits in $W^*$ is given by the following theorem, which is essentially due to Feldvoss and Nakano, but with some additions and corrections:
\begin{thm}\label{snart}
For any $j\in \mathbb{N}$, let $k^{(j)}\subseteq k$ denote a set of representatives for the equivalence classes of the equivalence relation on $k$ given by: $x\sim y \Leftrightarrow x^j=y^j$. A set of representatives for the orbits of height $i$ in $W^*$ is:
\begin{enumerate}
\item $\{ae_0'\mid a\in k^*\}$ if $i=1$
\item $\{e_{i-1}'\}$ if $0\leq i\leq p-3$ and $i$ is even
\item $\{e_{i-1}'+ae_{\frac{i-1}{2}}'\mid a\in k^{(2)}\}$ if $3\leq i\leq p-2$ and $i$ is odd
\item $\{e_{p-2}'+ae_{-1}'\mid a\in k^{(p-2)}\}$ if $i=p-1$ 
\end{enumerate}
Furthermore, the dimensions of the orbits are:
\begin{enumerate}
\item[1'.] $\mathrm{dim}(G(ae_0'))=1$
\item[2'.] $\mathrm{dim}(Ge_{i-1}')=i+1$ if $0\leq i\leq p-3$ and $i$ is even
\item[3'.] $\mathrm{dim}(G(e_{i-1}'+ae_{\frac{i-1}{2}}'))=i$ if $3\leq i\leq p-2$ and $i$ is odd
\item[4'.] $\mathrm{dim}(G(e_{p-2}'+ae_{-1}'))=p-1$  
\end{enumerate}
\end{thm}
\begin{proof}
Case 1 is analogous to 2 in Theorem $\ref{baner}$ and proved in exactly the same way, while case 2 was proved in $\cite{FN}$. Here it was also shown that if $3\leq i\leq p-2$ and $i$ is odd, then every character of height $i$ is in the orbit of some $e_{i-1}'+ae_{\frac{i-1}{2}}'$, but Feldvoss and Nakano actually gave incorrect results in cases 1 and 3, having seemingly overlooked the action of $T$ (this was already remarked in $\cite{R}$). Now assume $e_{i-1}'+ae_{\frac{i-1}{2}}'$ and $e_{i-1}'+a'e_{\frac{i-1}{2}}'$ are in the same orbit, for some $a,a'\in k$. So there exists $g=t\circ \sigma_{\varphi}\in G$ (with $t\in T$, $\sigma_{\varphi}\in U$) such that $g(e_{i-1}'+ae_{\frac{i-1}{2}}')=e_{i-1}'+a'e_{\frac{i-1}{2}}'$, or equivalently:
\begin{equation*}
\sigma_{\varphi}(e_{i-1}'+ae_{\frac{i-1}{2}}')=t^{i-1}e_{i-1}'+t^{\frac{i-1}{2}}a'e_{\frac{i-1}{2}}'
\end{equation*}
If the coefficients of $e_{i-2}',\dots,e_{\frac{i-1}{2}+1}'$ in $\sigma_{\varphi}(e_{i-1}'+ae_{\frac{i-1}{2}}')$ are all zero, then the coefficient of $e_{\frac{i-1}{2}}'$ must be $a$ (look at equation ($\ref{ny}$)). So we get $t^{i-1}=1$ and $a=t^{\frac{i-1}{2}}a'$. Raising this last equation to the second power yields $a^2=(a')^2$. If on the other hand this is true, i.e., $a'=\pm a$, then it is easily seen that $e_{i-1}'+ae_{\frac{i-1}{2}}'$ and $e_{i-1}'+a'e_{\frac{i-1}{2}}'$ are in the same orbit (in the case $a'=-a$ just find $t\in T$ satisfying $t^{-\frac{i-1}{2}}=-1$ and let it act on $e_{i-1}'+ae_{\frac{i-1}{2}}'$). This proves case 3, and case 4 is proved in the same way.  As for the dimension statements, the first two cases are straightforward, and one simply adapts the proof of Theorem 4.1 in $\cite{YS}$ for the last two. We will give an outline of the steps needed in case 3': Write $\chi=e_{i-1}'+ae_{\frac{i-1}{2}}'$, and let $G_\chi$ denote the stabilizer of $\chi$ in $G$. Then we have
\begin{equation*}
G_\chi=(G_\chi\cap T)\ltimes(G_\chi\cap U)
\end{equation*}
To prove this, it is enough to show that if $t^{-1}\in T$, $u\in U$ and $t^{-1}\circ u\in G_\chi$, then $t^{-1},u\in G_\chi$. Let $u(\chi)=e_{i-1}'+a_{i-2}e_{i-2}'+\dots +a_{-1}e_{-1}'$. We get $t^{-1}\circ u(\chi)=t^{i-1}e_{i-1}'+t^{i-2}a_{i-2}e_{i-2}'+\dots +t^{-1}a_{-1}e_{-1}'$, and since we assume $t^{-1}\circ u\in G_\chi$ we must have $a_{i-2},\dots,\hat{a}_{\frac{i-1}{2}},\dots,a_{-1}=0$. It now follows from equation ($\ref{synth}$) on the next page that $a_{\frac{i-1}{2}}=a$. So $u\in G_\chi$, which easily implies that $t^{-1}\in G_\chi$ too.
    
Setting $G_1=G_\chi\cap T$ and $G_2=G_\chi\cap U$, it can be shown that
\begin{equation*}
G_1=\{t\in T\mid t^\frac{i-1}{2}=1\}
\end{equation*}
\begin{equation*}
G_2=\{\sigma_\varphi\in U\mid \varphi(x)=x+b_{i+1}x^{i+1}+\dots+b_{p-1}x^{p-1}\}
\end{equation*}
But then $\mathrm{dim}(G_\chi)=\mathrm{dim}(G_1)+\mathrm{dim}(G_2)=p-i-1$, which implies
\begin{equation*}
\mathrm{dim}(G\chi)=\mathrm{dim}(G)-\mathrm{dim}(G_\chi)=i
\end{equation*}
\end{proof}

Now we determine the closures in $W^*$ of orbits of all heights except $p-1$: 

\begin{thm}\label{fu}
Let $a\in k$. We have:
\begin{enumerate}
\item $\overline{G(ae_0')}=G(ae_0')$
\item $\overline{Ge_{i-1}'}=\{\chi\in W^*\mid r(\chi)\leq i\}$ if $0\leq i\leq p-3$ and $i$ is even
\item $\overline{G(e_{i-1}'+ae_{\frac{i-1}{2}}')}=G(e_{i-1}'+ae_{\frac{i-1}{2}}')\cup \{\chi\in W^*\mid r(\chi)\leq i-2\}$ if $3\leq i \leq p-2$ and i is odd
\end{enumerate} 
\end{thm}  
\begin{proof} 
Cases 1 and 2 are trivial, so let us concentrate on case 3: Since the procedure is similar to the one we used to prove case 3 in Theorem $\ref{mopho}$ we will omit quite a few details. Write $s=\frac{i-1}{2}$ and $\sigma_\varphi^{-1}(e_{i-1}')=\sum_{j=-1}^{i-1}a_je_j'$. We grade the polynomial ring $k[U]=k[b_2,\dots,b_{p-1}]$ by setting $\mathrm{deg}(b_s)=s-1$. Then it is easily shown that $a_{-1},\dots,a_{i-2}\in k[U]$, and that $a_j$ is homogeneous of degree $i-j-1$. Proceeding exactly as in the proof of Proposition $\ref{c}$, we get: 
\begin{equation}\label{synth}
U(e_{i-1}'+ae_s')=\left\{\begin{pmatrix}
a_{-1}\\ \vdots\\a_{s-1}\\f(a_{s+1},\ldots,a_{i-2})+a\\a_{s+1}\\ \vdots\\a_{i-2} \\1\\0\ \vdots \\0
\end{pmatrix} \right\}
\end{equation}
where $a_{-1},\dots,\hat{a}_{2s},\dots,a_{i-2}\in k$ and $f\in k[X_{s+1},\dots,X_{i-2}]$ is a polynomial satisfying the following properties:
\begin{enumerate}
\item If $X_{s+1}^{\alpha_{s+1}}\dots X_{i-2}^{\alpha_{i-2}}$ is a monomial appearing in $f$ with nonzero coefficient, then 
\begin{equation*}
\sum_{j=s+1}^{i-2} (i-j-1)\alpha_j=s
\end{equation*}
\item The (usual) degree of $f$ is at most $s$, and the component in $f$ of degree $s$ is $cX_{i-2}^s$ for some $c\in k$. 
\end{enumerate}
One can now write $f=f_0+\dots+f_s$ ($f_j$ being the component in $f$ of degree $j$) and repeat the proof of Proposition $\ref{end}$ to get
\begin{equation*}
G(e_{i-1}'+ae_s')=\left\{ \begin{pmatrix}
b_{-1}\\ \vdots\\ b_{s-1} \\ \\ \sum_{j=0}^s\frac{f_j(b_{s+1},\dots,b_{i-2})}{b_{i-1}^{j-1}}\pm a\sqrt{b_{i-1}}\\ \\b_{s+1}\\ \vdots\\b_{i-1}\\0\\ \vdots \\0
\end{pmatrix} \right\}
\end{equation*} 
where $b_{-1},\dots,\hat{b}_s,\dots,b_{i-1}\in k$ and $b_{i-1}\neq 0$. It follows directly that
\begin{align*}
G(&e_{i-1}'+ae_s')=\\& \{x\in V((X_sX_{i-1}^{s-1}-cX_{i-2}^s-\sum_{j=0}^{s-1}X_{i-1}^{s-j}f_j)^2-a^2X_{i-1}^{2s-1})\mid x_{i-1}\neq 0\}
\end{align*}
If $c=0$ then $\overline{G(e_{i-2}')}\subsetneq \overline{G(e_{i-1}'+ae_s')}$, just as in the proof of Proposition $\ref{fini}$, but this is a contradiction since the dimension of the two varieties is the same (Theorem $\ref{snart}$). So $c\neq 0$ and the theorem follows easily.
\end{proof}
  
It remains only to determine the orbit closures of characters of height $p-1$. This case turns out to be the hardest, and the strongest statement we can prove is:
\begin{prop}\label{bo}
Either $\overline{Ge_{p-2}'}=Ge_{p-2}'\cup \{\chi\in W^*\mid r(\chi)\leq p-3\}$ or $\overline{Ge_{p-2}'}=Ge_{p-2}'\cup Ge_{p-3}'\cup \{\chi\in W^*\mid r(\chi)\leq p-3\}$. If the former is true, then 
\begin{equation}\label{fi}
\overline{G(e_{p-2}'+ae_{-1}')}=G(e_{p-2}'+ae_{-1}')\cup \{\chi\in W^*\mid r(\chi)\leq p-3\}
\end{equation}
for all $a\in k$.
\end{prop}
\begin{proof}
Applying the usual method, we get
\begin{align*}
&G(e_{p-2}'+ae_{-1}')=\\& \{x\in V((X_{-1}X_{p-2}^{p-2}-cX_{p-3}^{p-1}-q)^{p-2}-a^{p-2}X_{p-2}^{(p-2)(p-2)-1})\mid x_{p-2} \neq 0\}
\end{align*}
where $q=\sum_{j=1}^{p-2}X_{p-2}^{p-1-j}f_j$ and $f_j$ is a homogeneous polynomial of degree $j$. Furthermore, if $X_0^{\alpha_0}\dots X_{p-3}^{\alpha_{p-3}}$ is a monomial appearing in any of the $f_j$ with non-zero coefficient, then
\begin{equation}\label{gu}
\sum_{k=0}^{p-3}(p-2-k)\alpha_k=p-1
\end{equation}  
We would like to prove $c\neq 0$, but this time we cannot use a dimension argument (notice that the dimension of an orbit of height $p-2$ is \textit{one less} than the dimension of $\overline{G(e_{p-2}'+ae_{-1}')}$), nor look at the degree of the polynomials defining the orbits (as in Proposition $\ref{fini}$), to derive a contradiction. It is, however, possible to say something meaningful when $a=0$. In this case we have:
\begin{equation*}
Ge_{p-2}'=\{x\in V(X_{-1}X_{p-2}^{p-2}-cX_{p-3}^{p-1}-\sum_{j=1}^{p-2}X_{p-2}^{p-1-j}f_j)\mid x_{p-2} \neq 0\}
\end{equation*}
If $c\neq 0$, then the orbit closure is equal to the zero set of the polynomial on the right hand side, which is exactly $Ge_{p-2}'\cup \{\chi\in W^*\mid r(\chi)\leq p-3\}$, and the same argument for arbitrary $a$ gives us ($\ref{fi}$). Now assume $c=0$ and let l be the maximal number satisfying $f_l\neq 0$. Then we get:
\begin{equation*}
\overline{Ge_{p-2}'}=V(X_{-1}X_{p-2}^{l-1}-
\sum_{j=1}^lX_{p-2}^{l-j}f_j)
\end{equation*}   
The polynomial on the right hand side will be denoted by $f'$. It follows, as in the proof of Proposition $\ref{fini}$, that an orbit $G(e_{p-3}'+a'e_{\frac{p-3}{2}})$ is contained in $\overline{Ge_{p-2}'}$. Assume $a'\neq 0$. Note that $Ge_{p-2}$ is invariant under multiplication by scalars different from zero, and the same is then true for the closure. Since $b(G(e_{p-3}'+a'e_{\frac{p-3}{2}}))=G(e_{p-3}'+\sqrt{b}a'e_{\frac{p-3}{2}})$ for any $b\neq 0$ we infer that
\begin{equation*}
\bigcup_{b\neq 0}G(e_{p-3}'+be_{\frac{p-3}{2}})\subseteq \overline{Ge_{p-2}'} 
\end{equation*}
But the set on the left hand side is open in $\{\chi\in W^*\mid r(\chi)\leq p-2\}$, and since this set is irreducible we have
\begin{equation*}
\{\chi\in W^*\mid r(\chi)\leq p-2\}\subsetneq \overline{Ge_{p-2}'}
\end{equation*}
But the dimension of both these varieties is $p-1$ which is a contradiction. This means that $a'$ cannot be zero, and so $Ge_{p-3}'\subseteq \overline{Ge_{p-2}'}$. From the proof of Theorem $\ref{fu}$ we get 
\begin{equation*}
\overline{Ge_{p-3}'}=V(X_s X_{p-3}^{s-1}-c'X_{p-4}^s-\sum_{j=1}^{s-1}X_{p-3}^{s-j}g_j)
\end{equation*} 
where $s=\frac{p-3}{2}$, $c'\neq 0$ and the polynomial on the right hand side - let us call it $g$ - is irreducible. Also every monomial $X_0^{\alpha_0}\dots X_{p-3}^{\alpha_{p-3}}$ appearing in $g$ with non-zero coefficient satisfies $\sum_{k=0}^{p-3}(p-3-k)\alpha_k=s$. The inclusion $\overline{Ge_{p-3}'}\subseteq \overline{Ge_{p-2}'}$ implies that $g$ divides $f'(X_0,\dots ,X_{p-3},0)=f_l$, so we can write $f_l=gh$ for some polynomial $h$. Now grade the polynomial ring in the $X_j$ by letting the degree of $X_j$ be $p-2-j$. Then $f_l$ is homogeneous of degree $p-1$, and for a monomial $X_0^{\alpha_0}\dots X_{p-3}^{\alpha_{p-3}}$ appearing in $g$ we have
\begin{equation*}
\sum_{j=0}^{p-3}(p-2-j)\alpha_j=\sum_{j=0}^{p-3}(p-3-j)\alpha_j+\sum_{j=0}^{p-3}\alpha_j=2s=p-3
\end{equation*}
so $g$ is homogeneous of degree $p-3$ (note that we use the fact that $g$ is homogeneous of degree $s$ in the usual sense for the second equality). It follows that $h$ is homogeneous of degree $2$, which leaves only two possibilities: either $h=dX_{p-3}^2$ or $h=d'X_{p-4}$. In the second case we get $f_l=d'X_{p-4}g$ and
\begin{align*}
&\overline{Ge_{p-2}'}=\\& Ge_{p-2}'\cup \{x\in W^*\mid x_{p-2}=0,f_q(x_{-1},\dots,x_{p-3})=0\}=\\& Ge_{p-2}'\cup\overline{Ge_{p-3}'}\cup \{x\in W^*\mid x_{p-2}=x_{p-4}=0\} 
\end{align*}
But this is impossible: Consider the set $B=\{x\in W^*\mid x_{p-2}=0,x_{p-3}=b,x_{p-4}=0\}$ for some $b\neq 0$. This set would then be contained in $\overline{Ge_{p-2}'}$, and since characters of height $p-2$ in $\overline{Ge_{p-2}'}$ must be contained in $\overline{Ge_{p-3}'}$ we get $B\subseteq \overline{Ge_{p-3}'}$. But this means that $g$ becomes zero when inserting $b$ for $X_{p-3}$ and 0 for $X_{p-4}$, and this is clearly untrue. So we must have $f_l=dX_{p-3}^2g$, and a calculation similar to the previous one gives us:
\begin{equation*}
\overline{Ge_{p-2}'}=Ge_{p-2}'\cup Ge_{p-3}'\cup\{\chi\in W^*\mid r(\chi)\leq p-3\}
\end{equation*}  
\end{proof}
We can deduce the following interesting corollary from our results:
\begin{corollary}
The algebra of invariants $k[W^*]^G$ is trivial, i.e., $k[W^*]^G=k$.
\end{corollary}
\begin{proof}
Let $f\in k[W^*]^G$ and $\chi\in W^*$ be arbitrary. We want to prove that $\chi$ is contained in an orbit closure that also contains zero. Since $f$ is constant on orbit closures this would imply $f(\chi)=f(0)$ and $f$ is constant on all of $W^*$. It follows from what we have shown so far that $\chi$ is contained in an orbit closure $\overline{G\chi'}$ where $\chi'$ has height $p-3$, $p-2$ or $p-1$. In the first two cases we already know from Theorem $\ref{fu}$ that $0\in \overline{G\chi'}$, so let us assume that $\chi'=e_{p-2}'+ae_{-1}'$ for some $a\in k$. It is possible to derive the facts we need from the proof of Proposition $\ref{bo}$, but since this is quite technical, we present a more direct approach here (suggested by J. C. Jantzen): Set $s=\frac{p-1}{2}$ and consider the automorphism $\varphi$ of $A(1)$ defined by $\varphi(x)=x+bx^{s+1}$ for some $b\in k$. An easy calculation shows that
\begin{equation*}
\varphi(x^j)=x^j+jbx^{s+j}
\end{equation*} 
\begin{equation*}
\varphi^{-1}(x^j)=x^j-jbx^{s+j}
\end{equation*}
for $1\leq j\leq p-1$. Now we get:
\begin{align*}
\sigma_\varphi(e_j)(x)=&\varphi\circ e_j\circ\varphi^{-1}(x)=\\&\varphi\circ e_j(x-bx^{s+1})=\varphi(x^{j+1}-(s+1)bx^{s+j+1})
\end{align*}
Recall that $x^r=0$ for $r\geq p$. So if $s\leq j\leq p-2$, then $\sigma_\varphi(e_j)(x)=\varphi(x^{j+1})=x^{j+1}$. Similarly, we get $\sigma_\varphi(e_j)(x)=\varphi(x^{j+1}-(s+1)bx^{s+j+1})=x^{j+1}-(s-j)bx^{s+j+1}$ if $0\leq j<s$, and finally
\begin{equation*}
\sigma_\varphi(e_{-1})(x)=\varphi(1-(s+1)bx^s)=1-(s+1)bx^s-s(s+1)b^2x^{p-1}
\end{equation*}  
Using these calculations - and formula ($\ref{diff}$) - we get:
\begin{equation*}
\sigma_\varphi(e_j)=\begin{cases} e_j & \text{if $s\leq j\leq p-2$} \\ e_j-(s-j)be_{s+j} & \text{if $0\leq j<s$} \\ e_{-1}-(s+1)be_{s-1}-s(s+1)b^2e_{p-2} & \text{if $j=-1$} \end{cases}
\end{equation*}
Now it is easy to check that
\begin{equation*}
\sigma_\varphi^{-1}(e_{p-2}')=e_{p-2}'-be_{s-1}'-s(s+1)b^2e_{-1}'
\end{equation*}
If we choose $b$ so that $s(s+1)b^2=a$ we get:
\begin{equation*}
\sigma_\varphi^{-1}(e_{p-2}'+ae_{-1}')=e_{p-2}'-be_{s-1}'
\end{equation*}
As will be seen, the whole point is that no $e_{-1}'$-term appears on the right-hand side. Define a morphism $\psi:\mathbb{A}^1\longrightarrow W^*$ by $\psi(t)=t^{p-2}e_{p-2}'-t^{s-1}be_{s-1}'$. We have $\psi(\mathbb{A}^1\backslash \{0\})=T(\sigma_\varphi^{-1}(\chi'))\subseteq G\chi'$, and since the image of the closure is contained in the closure of the image for any continuous map, we conclude that $0=\psi(0)\in \overline{G\chi'}$.
\end{proof}

Premet states in his paper $\cite{P}$ that "It would be of interest to describe the invariants of the coadjoint representation $W(n)^*$ of $G$", and this is a first step in that direction. 

\begin{proof}[Acknowledgements]
First and foremost I would like to thank my advisor Professor Jens Carsten Jantzen for his many insightful comments and suggestions. I would also like to thank my fellow Ph.d. students Troels Bak Andersen and Jens-Jakob Kratmann Nissen for their support and for numerous interesting and enlightening discussions. 
\end{proof}

\end{document}